\documentclass[12pt]{article}

\usepackage[margin=1in]{geometry}
\usepackage{amsmath,mathrsfs}
\usepackage{amssymb,amsthm}
\usepackage{graphicx}
\usepackage{paralist}
\usepackage{mathdots}

\newcommand{\real}{\mathbb{R}}
\newcommand{\bs}[1]{\mathbf{#1}}

\newtheorem{theorem}{Theorem}[section]
\newtheorem{corollary}{Corollary}[section]

\newtheorem{lemma}{Lemma}[section]
\theoremstyle{definition}
\newtheorem{definition}{Definition}[section]

\title{Anti-regular graphs with loops and their spectrum}
\author{Cesar O. Aguilar\footnote{Department of Mathematics, State University of New York, Geneseo, aguilar@geneseo.edu}}

\begin{document}

\maketitle

\begin{abstract}
An anti-regular graph is a graph whose degree sequence has only two repeated entries.  It is known that for each integer $n\geq 1$ there are two non-isomorphic anti-regular graphs on $n$ vertices and such graphs are examples of threshold graphs.  We show that if one allows loops in a graph, then there are two non-isomorphic graphs whose degree sequence contains all distinct entries and we compute closed-form expressions for the adjacency eigenvalues of such graphs.
\end{abstract}

\section{Introduction}

Let $G=(V,E)$ be an $n$-vertex simple graph, that is, a graph without loops or multiple edges.  If $n\geq 2$ then it is well-known that $G$ contains at least two vertices of equal degree.  If all vertices of $G$ have equal degree then $G$ is called a \textit{regular} graph.  It is then natural to say that $G$ is \textit{anti-regular} if $G$ has only two vertices of equal degree.  If $G$ is anti-regular then the complement graph $\overline{G}$ is easily seen to be anti-regular, and moreover, if $G$ is connected then $\overline{G}$ is disconnected.  It was shown in \cite{MB-GC:67} that up to isomorphism, there is only one connected anti-regular graph on $n$ vertices and that its complement is the unique disconnected $n$-vertex anti-regular graph.  We denote by $A_n$ the unique connected anti-regular graph on $n\geq 2$ vertices.  

An anti-regular graph is an example of a \textit{threshold graph} \cite{NM-UP:95}.  Threshold graphs were first studied by Chv\'{a}tal and Hammer \cite{VC-PH:77} (see also \cite{PH-YZ:77}) in the problem of aggregating a set of linear inequalities in integer programming problems.  Threshold graphs have been studied extensively and have applications in parallel computing, linear programming, resource allocation problems, and psychology \cite{NM-UP:95}.  Threshold graphs are also employed as reasonable models of real-world social networks, friendship networks, peer-to-peer networks, and networks of computer programs \cite{PD-SH-SJ:08, YI-NK-NO:09}.

Within the family of threshold graphs, the anti-regular graph $A_n$ plays a prominent role.  For instance, $A_n$ is uniquely determined by its \textit{independence polynomial} \cite{VL-EM:12}.   It was shown in \cite{RM:03} that $A_n$ is \textit{universal for trees}, that is, every tree graph on $n$ vertices is isomorphic to a subgraph of $A_n$.   Also, the eigenvalues of the Laplacian matrix of $A_n$ are all distinct integers and the missing eigenvalue from $\{0,1,\ldots, n\}$ is $\lfloor (n+1)/2\rfloor$.  In \cite{EM:09}, the characteristic and matching polynomial of $A_n$ are studied and several recurrence relations are obtained for these polynomials, along with some spectral properties of the adjacency matrix of $A_n$.  Recently in \cite{CA-JL-EP-BS:18}, a fairly complete characterization of the eigenvalues of the adjacency matrix of $A_n$ was obtained.  In particular, it was shown that the interval $\Omega=[\frac{-1-\sqrt{2}}{2}, \frac{-1+\sqrt{2}}{2}]$ contains only the eigenvalues $-1$ or $0$ (depending on wether $n$ is even or odd) and that as $n\rightarrow\infty$, the eigenvalues are symmetric about $-\frac{1}{2}$ and are dense in $(-\infty,\frac{-1-\sqrt{2}}{2}]\cup\{0,-1\}\cup [\frac{-1+\sqrt{2}}{2},\infty)$.  It is conjectured in \cite{CA-JL-EP-BS:18} that in fact the only eigenvalues of any threshold graph contained in $\Omega$ are $-1$ and/or $0$, which if true would improve the forbidden interval bound obtained in \cite{DJ-VT-FT:15}, and furthermore would show that the eigenvalues of $A_n$ are optimal when viewed within the family of threshold graphs.

In this note, we consider the anti-regular graph with loops.  Random threshold graphs with loops have been considered in \cite{AB-AS:07, PD-SH-SJ:08, YI-NK-NO:09} where asymptotic properties such as the spectral distribution and the rank of the adjacency matrix were studied.    In this paper, we show that if one allows loops in a graph then there are two non-isomorphic graphs whose vertex degrees are \textit{all} distinct.    Not surprisingly, these anti-regular graphs with loops have as their underlying simple graph the anti-regular simple graphs.  Exploiting the structure of the adjacency matrix of an anti-regular graph, we perform similarity transformations that lead to closed-form expressions for the eigenvalues of the adjacency matrix of the anti-regular graphs with loops.  As in the case of simple threshold graphs (i.e., with no loops), we conjecture that the spectral properties of anti-regular graph with loops will play an important role within the class of threshold graphs with loops.  This will be considered in a future investigation. 

\section{The anti-regular graph with loops}

By a \textit{graph with loops} we mean a pair $G=(V,E)$ where $V=\{v_1,v_2,\ldots,v_n\}$ is a finite vertex set and $E$ is a set containing 2-element multi-sets of $V$, that is, $E \subset \{ \{u,v\}\;|\; u, v \in V\}$.  Hence, the only extra freedom we allow from simple graphs is that some vertices may contain at most one self-loop.  The degree of a vertex in a graph with loops is the number of edges incident to the vertex with loops counted once.  As usual, if $G$ has no loops then we will call $G$ a \textit{simple graph}.

A graph with loops $G=(V,E)$ is a \textit{threshold} graph if there exists a vertex weight function $w:V\rightarrow [0,\infty)$ and $t\geq 0$, called the \textit{threshold}, such that for every $X\subset V$, $X$ is an independent set if and only if $\sum_{v\in X} w(v) \leq t$.  In resource allocation problems, the weight $w(v_i)$ is interpreted as the amount of resources used by vertex $v_i$ and thus a subset of the vertices $X$ is an admissible set of vertices if the total amount of resources needed by $X$ is no more than the allowable threshold $t$.  Of the several characterizations of threshold graphs \cite{NM-UP:95}, the most convenient is a one-to-one correspondence between threshold graphs and binary sequences.  Given a binary sequence $b=(b_1,b_2,\ldots,b_n)\in\{0,1\}^n$, we use a recursive process using the join and union graph operations to construct a $n$-vertex graph $G=G(b)$ as follows.  Start with a single vertex $v_1$ and add the loop $\{v_1,v_1\}$ if and only if $b_1=1$.  Then add a new vertex $v_2$ and connect $v_2$ with all existing vertices $\{v_1,v_2\}$ if and only if $b_2=1$.  Iteratively, if the vertices $\{v_1, v_2,\ldots,v_k\}$ have been added, add vertex $v_{k+1}$ and connect it to all existing vertices $\{v_1,v_2,\ldots,v_{k+1}\}$ if and only if $b_{k+1} = 1$.  If $b_{j} = 1$ we call $v_j$ a \textit{dominating vertex}; in this case there is a loop at $v_j$ and $v_j$ is adjacent to all vertices $\{v_1,\ldots,v_{j-1}\}$.  If $b_j=0$ then $v_j$ is an \textit{isolated vertex}.  Thus, an $n$-vertex threshold graph is connected if and only if $b_n=1$.  It is readily seen that the adjacency matrix of $G=G(b)$ is 
\[
A(G) = \begin{bmatrix}
b_1 & b_2 & b_3 & \cdots & \cdots & b_n\\
b_2 & b_2 & b_3 & \cdots & \cdots & b_n\\
b_3 & b_3 & b_3 & \cdots & \cdots & \vdots\\
\vdots & \vdots & \vdots & \ddots & \vdots & b_n\\
\vdots & \vdots & \vdots & \vdots & \ddots & \vdots\\
b_n & b_n & \cdots & \cdots & b_n & b_n
\end{bmatrix} .  
\]

\begin{definition}
A graph $G=(V,E)$ with loops is called an \textit{anti-regular graph with loops} if all vertices of $G$ have distinct degrees.
\end{definition}

For any graph $G=(V,E)$ we denote its degree sequence by 
\[
d(G)=(d_1,d_2,\ldots,d_n)
\]
and we assume without loss of generality that $\textup{deg}(v_i) = d_i$ for $i\in\{1,2,\ldots,n\}$ and $d(G)$ is a non-decreasing sequence.  The number of non-decreasing sequences $(d_1,d_2,\ldots,d_n)$ with no repeated entries and such that $d_i \in \{0,1,\ldots, n\}$ is clearly $n+1$.  Each such sequence is characterized by the number in $\{0,1,\ldots,n\}$ not present in the sequence.  For $n=1$ there are two anti-regular graphs with loops on $n$ vertices; one is the null graph with one vertex which we denote by $H_1$ and the second is the graph with one vertex and one loop which we denote by $G_1$.  Thus, $d(H_1) = (0)$ and $H_1$ is disconnected, and $d(G_1) = (1)$ and $G_1$ is connected.  If $G$ is an anti-regular graph with loops on $n\geq 2$ vertices with degree sequence $d(G) = (d_1,d_2,\ldots,d_n)$ then if $d_n = n$ then necessarily $d_1 = 1$ and if $d_1=0$ then necessarily $d_n = n-1$.  We have therefore proved the following.  
\begin{lemma}
If $G$ is an anti-regular graph with loops on $n\geq 1$ vertices then either $d(G) = (1,2,\ldots,n)$ or $d(G) = (0,1,\ldots, n-1)$.  
\end{lemma}

We now characterize the anti-regular graphs with loops.  The proof is similar to that of Theorem 2 in \cite{MB-GC:67} where simple graphs are considered.

\begin{theorem}
For each $n\geq 1$ the following hold: 
\begin{enumerate}[(i)]
\item There are two non-isomorphic anti-regular graphs with loops on $n$ vertices, which we denote by $G_n$ and $H_n$.  
\item The graphs $G_n$ and $H_n$ are complementary, and one of them is connected and the other is disconnected.  
\item If $G_n$ is the connected one, then $d(G_n) = (1,2,\ldots, n)$ and \\ $d(H_n) = (0,1,\ldots, n-1)$.
\item Both $G_n$ and $H_n$ are threshold graphs with loops whose binary sequences are alternating.
\item The binary sequence of $G_n$ begins with $n\bmod 2$.
\item The binary sequence of $H_n$ begins with $(n+1)\bmod 2$ and is obtained by flipping the bits in the binary sequence of $G_n$.  
\end{enumerate}
\end{theorem}
\begin{proof}
The proof is by induction.  The base case $n=1$ has been proved above. By induction, assume that we have proved the claim for some $n\geq 1$.  Hence, $G_n$ and $H_n$ denote the anti-regular graphs with loops on $n$ vertices, where $G_n$ is connected, $H_n$ is the complement of $G_n$ and is disconnected, $d(G_n) = (1,2,\ldots,n)$, and $d(H_n) = (0,1,\ldots, n-1)$.  The binary sequence of $G_n$ begins with $n \bmod 2$ and $H_n$ has binary sequence beginning with $(n+1)\bmod 2$.  Let $G_{n+1}$ be the threshold graph obtained by adding the dominating vertex $v_{n+1}$ to $H_n$.  Then clearly $d(G_{n+1}) = (1,2,\ldots,n,n+1)$, $G_{n+1}$ is connected, and $G_{n+1}$ has binary sequence beginning with $(n+1)\bmod 2$.  On the other hand, let $H_{n+1}$ be the threshold graph obtained by adding the isolated vertex $v_{n+1}$ to $G_n$.  Then clearly $H_{n+1}$ is disconnected and after a relabelling of the vertices we have $d(H_{n+1}) = (0,1,2,\ldots, n)$.  The binary sequence of $H_{n+1}$ begins with $n \bmod 2 \equiv (n+1)+1\bmod 2$.  By construction and the induction hypothesis, $H_{n+1}$ is the complement of $G_{n+1}$.  If now $G$ is an anti-regular graph with loops on $n+1$ vertices with $d(G) = (1,2,\ldots, n+1)$ then the graph $G - v_{n+1}$ has degree sequence $(0,1,\ldots,n-1)$ and therefore $G$ is $G_{n+1}$ by the induction hypothesis.  On the other hand, if $G$ has degree sequence $d(G) = (0,1,\ldots,n)$ then $G - v_1$ has degree sequence $(1,2,\ldots,n)$ and thus $G$ is $H_{n+1}$ by the induction hypothesis.
\end{proof}

\section{The eigenvalues of anti-regular graphs with loops}
In this section we show that, contrary to the case of the simple anti-regular graphs \cite{CA-JL-EP-BS:18}, the eigenvalues of the adjacency matrix of the anti-regular graph with loops have closed-form formulas.  To that end we recall that $G_n$ denotes the connected and $H_n$ denotes the disconnected anti-regular graph with loops on $n\geq 1$ vertices.  We begin with the following.
\begin{lemma}
Let $G_n$ be the connected anti-regular graph with loops on $n\geq 1$ vertices.  Then the adjacency matrix $A(G_n)$ is permutation similar to the Hankel matrix
\[
M_n = 
\begin{bmatrix}
&&&&1\\
&&&\iddots &1\\
&&\iddots&\iddots & \vdots\\
&\iddots&\iddots&&\vdots\\
1&1&\cdots&\cdots&1
\end{bmatrix}.
\]
\end{lemma}
\begin{proof}
We first note that for $n\geq 3$
\[
M_n = \begin{bmatrix} 0 & \bs{0}^T & 1 \\[2ex] \bs{0} & M_{n-2} & \bs{1}\\[2ex] 1 & \bs{1}^T & 1 \end{bmatrix}
\]
where $M_1=1$.  The proof of the lemma is by strong induction.  For $n=1$, we have $A(G_1) = M_1$ and for $n=2$ we have $A(G_2) = M_2$, and so the claim is trivial for these cases.  Assume by strong induction that the claim is true for all $k\in\{1,\ldots, n\}$ where $n\geq 2$.  Using the labelling of the vertices corresponding to the binary sequence of $G_{n+1}$ we have 
\begin{equation}\label{eqn:A-Gn1}
A(G_{n+1}) = \begin{bmatrix} A(H_n) & \bs{1}\\[2ex] \bs{1}^T & 1\end{bmatrix}.
\end{equation}  
Similarly, using the labelling of the vertices corresponding to the binary sequence of $H_n$ we have 
\[
A(H_n) = \begin{bmatrix} A(G_{n-1}) & \bs{0} \\[2ex] \bs{0}^T & 0 \end{bmatrix}.
\]
It is not hard to see that $A(H_n)$ is permutation similar to 
\[
\begin{bmatrix} 0 & \bs{0}^T \\[2ex] \bs{0} & A(G_{n-1})\end{bmatrix}.
\]
By the induction hypothesis, $A(G_{n-1})$ is permutation similar to $M_{n-1}$ and therefore $A(H_n)$ is permutation similar to
\[
\begin{bmatrix} 0 & \bs{0}^T \\[2ex] \bs{0} & M_{n-1}\end{bmatrix}.
\]   
Therefore, from \eqref{eqn:A-Gn1} it follows that $A(G_{n+1})$ is permutation similar to
\[
\begin{bmatrix} 0 & \bs{0}^T & 1\\[2ex] \bs{0} & M_{n-1} & \bs{1} \\[2ex] 1 & \bs{1}^T & 1\end{bmatrix} = M_{n+1}.
\]

\end{proof}

We now determine the eigenvalues of $M_n$, and thus also of $A(G_n)$.  First of all, it is straightforward to show by induction that
\[
\det(M_n) = -\det(M_{n-2}) = (-1)^{\lfloor (n-1)/2\rfloor}.
\]
More importantly, by inspection we have
\[
M_n^{-1} = \begin{bmatrix}
&&&-1&1\\
&&\iddots&\iddots&\\
&\iddots&\iddots&&\\
-1&\iddots&&&\\
1&&&&
\end{bmatrix}.
\]
If $D=\textup{diag}(1,-1,1,\ldots,(-1)^{n+1})$ then $D^{-1} = D$ and it is straightforward to verify that
\[
Y_n = D M_n^{-1} D = 
\begin{bmatrix}
&&&1&1\\
&&\iddots&\iddots&\\
&\iddots&\iddots&&\\
1&\iddots&&&\\
1&&&&
\end{bmatrix}.
\]
If $n$ is even let $\sigma:\{1,2,\ldots,n\}\rightarrow\{1,2,\ldots,n\}$ be the permutation
\[
\sigma = \begin{pmatrix} 1 & 2 & 3 & \cdots & \frac{n}{2} & \frac{n}{2}+1 & \frac{n}{2}+2 & \cdots & n-1 & n\\[2ex]
2 & 4 & 6 & \cdots & n & n-1 & n-3 & \cdots & 3 & 1\end{pmatrix},
\]
that is, $\sigma(k) = 2k$ for $k=1,2,\ldots,\frac{n}{2}$ and $\sigma(\frac{n}{2}+j) = n - (2j-1)$ for $j=1,2,\ldots,\frac{n}{2}$.  If $n$ is odd then let $\sigma:\{1,2,\ldots,n\}\rightarrow\{1,2,\ldots,n\}$ be the permutation
\[
\sigma = \begin{pmatrix} 1 & 2 & 3 & \cdots & \frac{n-1}{2} & \frac{n+1}{2} & \frac{n+1}{2}+1 & \cdots & n-1 & n\\[2ex]
2 & 4 & 6 & \cdots & n-1 & n & n-2 & \cdots & 3 & 1\end{pmatrix}.
\]
Let $P$ be the corresponding permutation matrix for $\sigma$ such that $Px=(x_{\sigma(1)}, x_{\sigma(2)},\ldots,x_{\sigma(n)})$ for $x\in\real^n$.  A straightforward proof by induction shows that
\[
X_n = P^T Y_n P = (-1)^{n+1}
\begin{pmatrix}
0&1&&&\\
1&\ddots&\ddots&&\\
&\ddots&\ddots&\ddots&\\
&&\ddots&0&1\\
&&&1&1
\end{pmatrix}.
\]
The eigenvalues of $X_n$ are known explicitly \cite[Theorem 2]{WY:05} and are given by
\[
\lambda_j = -2 (-1)^n \cos\left(\frac{2j-1}{2n+1}\pi\right),\quad j=1,2,\ldots,n.
\]
As a consequence, we obtain the following.
\begin{theorem}
Let $G_n$ be the connected anti-regular graph with loops on $n\geq 1$ vertices.  The eigenvalues $\lambda_1,\lambda_2,\ldots,\lambda_n$ of the adjacency matrix of $G_n$ are
\[
\lambda_j = \frac{(-1)^{n+1}}{2\cos\left(\frac{2j-1}{2n+1}\pi\right)},\quad j=1,2,\ldots,n.
\]
Consequently, the eigenvalues of $G_n$ are simple.
\end{theorem}

Let $\sigma(n)$ denote the set of the eigenvalues of $G_n$.  Then the eigenvalues of $H_n$ are therefore $\{0\} \cup \sigma(n-1)$, where $\sigma(0)=\emptyset$.  Finally, using the closed-form expressions for the eigenvalues of $G_n$ the following is immediate.

\begin{corollary}
Let $\sigma(n)$ denote the eigenvalues of the connected anti-regular graph with loops on $n\geq 1$ vertices.  Then the closure of $\bigcup_{n\geq 1} \sigma(n)$ is $(-\infty,-\frac{1}{2}] \cup [\frac{1}{2},\infty)$.
\end{corollary}

\section{Acknowledgements}
The authors acknowledge the support of the National Science Foundation under Grant No. ECCS-1700578.

\bigskip
\baselineskip 1.25em



\begin{thebibliography}{99}

\bibitem{AB-AS:07}
A. Bose, S. Arnab
\newblock On asymptotic properties of the rank of a special random adjacency matrix.
\newblock {\em Electronic Communications in Probability}, 12 (2007): 200--205.

\bibitem{CA-JL-EP-BS:18}
Aguilar, C.O., Lee, J.Y., Piato, E., Schweitzer, B.
\newblock Spectral characterizations of anti-regular graphs.
\newblock {\em Linear Algebra Appl.} 557 (2018), 84--104.

\bibitem{YA-GC-FC-PE-RG-OO:87}
Alavi, Y., Chartrand, G., Chung, F., Erd{\"o}s, P., Graham, R., Oellermann, O.
\newblock Highly irregular graphs.
\newblock {\em J. Graph Theory} 11 (1987), no. 2, 235--249.

\bibitem{MB-GC:67}
Behzad, M., Chartrand, G.
\newblock No graph is perfect.
\newblock {\em Amer. Math. Monthly} 74 1967 962--963.

\bibitem{GC-PE-OO:88}
Chartrand, G., Erd\"{o}s, P., Oellermann, O.
\newblock How to define an irregular graph.
\newblock {\em College Math. J.} 19 (1988), no. 1, 36--42.

\bibitem{VC-PH:77}	
Chv\'{a}tal, V., Hammer, P.L.
\newblock Aggregation of Inequalities in Integer Programming.
\newblock {\em Ann. of Discrete Math.}, 1 (1977), 145--162.

\bibitem{PD-SH-SJ:08}
P. Diaconis, H. Susan Holmes, and S. Janson 
\newblock Threshold graph limits and random threshold graphs.
\newblock {\em Internet Mathematics}, 5, no. 3 (2008): 267--320.

\bibitem{PH-YZ:77}
Henderson, P.B., Zalcstein, Y.
\newblock A graph-theoretic characterization of the PV class of synchronizing primitives.
\newblock {\em SIAM J. Comput.} 6 (1977), no. 1, 88--108.

\bibitem{YI-NK-NO:09}
Y. Ide, N. Konno, and N. Obata 
\newblock Spectral properties of the threshold network model.
\newblock {\em Internet Mathematics} 6, no. 3 (2009): 173--187.

\bibitem{DJ-VT-FT:15}
D. Jacobs, V. Trevisan, and F. Tura.
\newblock Eigenvalues and energy in threshold graphs.
\newblock {\em Linear Algebra and its Applicaitons}, 465: 412--425, 2015.

\bibitem{VL-EM:12}
V.E. Levit and E. Mandrescu.
\newblock On the Independence Polynomial of an Antiregular Graph.
\newblock {\em Carpathian Journal of Mathematics},  28(2): 279--288, 2012.

\bibitem{NM-UP:95}
Mahadev, N., Peled, U.N.
\newblock Threshold graphs and related topics.
\newblock {Ann. of Discrete Math.}, 56., 1995.

\bibitem{RM:01}
Merris, R.
\newblock {\em Graph Theory}.
\newblock John Wiley \& Sons, 2001.

\bibitem{RM:03}
Merris, R.
\newblock Antiregular graphs are universal for trees.
\newblock {\em Univ. Beograd. Publ. Elektrotehn. Fak. Ser. Mat.} 14 (2003), 1--3.

\bibitem{EM:09}
E. Munarini.
\newblock Characteristic, admittance, and matching polynomials of an antiregular graph.
\newblock {\em Applicable Analysis and Discrete Mathematics}, 3(1):157--176, 2009.

\bibitem{WY:05}
Yueh, W.C.
\newblock Eigenvalues of several tridiagonal matrices.
\newblock {\em Appl. Math. E-Notes.} 5 (2005), 66--74.



\end{thebibliography}
\end{document}